\documentclass[reqno,a4paper,12pt]{amsart}
\usepackage[utf8]{inputenc}
\usepackage{amsthm}
\usepackage[font=footnotesize]{caption}
\usepackage{color}
\usepackage[margin=2.5cm,top=3.5cm,bottom=3cm,footskip=1cm,headsep=1cm]{geometry}

\title[Galerkin schemes for time-dependent radiative transfer]{A class of Galerkin schemes for \\time-dependent radiative transfer}
\author[H. Egger]{Herbert Egger$^\dag$}
\author[M. Schlottbom]{Matthias Schlottbom$^\ddag$}
\thanks{$^\dag$Numerical Analysis and Scientific Computing, Department of Mathematics, TU Darmstadt, Dolivostr. 15, 64293 Darmstadt, Germany. 
Email: {\tt egger@mathematik.tu-darmstadt.de}
}
\thanks{$^\ddag$ Institute for Computational and Applied Mathematics,
University of Münster, Einsteinstr. 62, 48149 M\"unster, Germany.
Email: {\tt schlottbom@uni-muenster.de}
}

\def\r{r}
\def\s{s}
\def\t{t}
\def\n{n}
\def\sn{\s \cdot \n}
\def\sgrad{\s \cdot \nabla}

\def\S{\mathcal{S}}
\def\R{\mathcal{R}}
\def\A{\mathcal{A}}
\def\C{\mathcal{C}}
\def\D{\mathcal{D}}
\def\H{\mathcal{H}}

\def\Y{Y}

\def\RR{\mathbb{R}}

\def\WW{\mathbb{W}}
\def\VV{\mathbb{V}}

\def\XX{\mathbb{X}}

\newtheorem{theorem}{Theorem}[section]
\newtheorem{lemma}[theorem]{Lemma}
\newtheorem{problem}[theorem]{Problem}

\theoremstyle{definition}
\newtheorem{remark}[theorem]{Remark}

\begin{document}

\begin{abstract}
The numerical solution of time-dependent radiative transfer problems is challenging, both, 
due to the high dimension as well as the anisotropic structure of the underlying integro-partial differential equation.
In this paper we propose a general framework for designing numerical methods for time-dependent radiative transfer 
based on a Galerkin discretization in space and angle combined with appropriate time stepping schemes. 
This allows us to systematically incorporate boundary conditions and to preserve basic properties
like exponential stability and decay to equilibrium also on the discrete level. 
We present the basic a-priori error analysis and provide abstract error estimates that cover a wide class of methods.
The starting point for our considerations is to rewrite the radiative transfer problem as a system 
of evolution equations which has a similar structure like first order hyperbolic systems in acoustics 
or electrodynamics. This analogy allows us to generalize the main arguments of the numerical analysis for such 
applications to the radiative transfer problem under investigation.
We also discuss a particular discretization scheme based on a truncated spherical harmonic expansion 
in angle, a finite element discretization in space, and the implicit Euler method in time. 
The performance of the resulting mixed PN-finite element time stepping scheme is demonstrated by computational results.
\end{abstract}

\maketitle



{\footnotesize




\section{Introduction} \setcounter{equation}{0}

We consider the time-dependent mono-energetic radiative transfer equation
\begin{align*}
& \partial_t \phi(\r,\s,\t) + \s \cdot\nabla \phi(\r,\s,\t) + \sigma_t(\r) \phi(\r,\s,\t) 
\\&\qquad \qquad \qquad \qquad \qquad 
=  \int_{\S^2} k(\r,\s\cdot \s') \phi(\r,\s',\t)d\s'+q(\r,\s,\t),
\end{align*}
which describes the streaming of particles through a domain
and their scattering and absorption by the background medium.
The particle density $\phi$ here is a function of position $\r$ in space, 
direction $\s$ of propagation, and time $\t$. The total cross-section $\sigma_t$
and the scattering kernel $k$ characterize the interaction with the medium, 
and $q$ represents a source density. 
Integro-partial differential equations of similar form arise in various fields, 
e.g. in astrophysics, meteorology, nuclear spectroscopy and engineering, or in biomedical imaging; 
see \cite{CaseZweifel67,Cercignani88,Davison57,WangWu07} for some particular applications.

In most practical situations, the radiative transfer equation cannot be solved analytically
but only by appropriate numerical methods. The typical approach towards the 
numerical solution consists of two steps: a semi-discretization with respect to angle leading, e.g., to the well-known $S_N$ and $P_N$ approximations \cite{CaseZweifel67,DuderstadtMartin79}, and a subsequent discretization in space and time by finite difference or finite element methods and appropriate time stepping schemes.
We refer to \cite{Asadzadeh86,JohnsonPitkaranta83,Kanschat09,KloseHielscher08,ReedHill73} for algorithms based on $S_N$ type approximations, and to \cite{AdamsLarsen02,BrunnerHolloway2005,FraKlaLarYas07,event,starmap} for approximations of $P_N$ type. 
Corresponding discretization strategies for stationary radiative transfer problems can be found for instance in \cite{Ackroyd98,EggerSchlottbom12,KophaziLathouwers2015,LewisMiller84,ManResSta00,MarchukLebedev86,WidHiptSchw08}. 
We refer to \cite{Ackroyd98,LewisMiller84,MarchukLebedev86} for an introduction to various discretization approaches and further references. 

We consider a discretization framework for time-dependent radiative transfer based on the splitting $\phi = \phi^+ + \phi^-$ of the density into even and odd functions of the angle \cite{Vladimirov61}.
Numerical methods based on such splittings have been investigated in \cite{Ackroyd98,EggerSchlottbom12,LewisMiller84} for stationary and in \cite{AydinOliveiraGoddard04,starmap,event} for time-dependent radiative transfer problems.
In this paper, we use the mixed variational framework of \cite{EggerSchlottbom12} 
to rigorously rewrite the radiative transfer equation as a coupled system 
of operator equations
\begin{align*}
\begin{pmatrix} 
\partial_t \phi^+  \\ \partial_t \phi^-
\end{pmatrix}
+ 
\begin{pmatrix} 
0 & -\A' \\ \A & 0 
\end{pmatrix}
\begin{pmatrix} 
\phi^+  \\ \phi^-
\end{pmatrix}
+
\begin{pmatrix} 
\C + \R & 0 \\ 0 & \C 
\end{pmatrix}
\begin{pmatrix} 
\phi^+  \\ \phi^-
\end{pmatrix}
=
\begin{pmatrix} 
q^+  \\ q^-
\end{pmatrix}
\end{align*}
for the even and odd part of the density, respectively.
The operator $\A \phi = \sgrad \phi$ denotes the directional derivatives, 
$\A'$ is the adjoint, and $\R$ stems from the incorporation of the boundary conditions. 
The collision operator $\C$ defined by
$$
\C \phi(\r,\s) = \sigma_t(\r)\phi(\r,\s)- \int_{\S^2} k(\r,\s \cdot \s') \phi(\r,\s') d\s'
$$ 
combines the effects of scattering and absorption. 
At this point one can observe that systems of similar abstract form also arise in the modeling of acoustic, electro-magnetic, or elasto-dynamic wave phenomena; see e.g. \cite{CohenMonk99,Dupont73,Geveci88,Makridakis92}.
For such applications, the systematic discretization and numerical analysis 
is already well-understood; we refer to \cite{Joly03} for an overview and further references.
%
%
A careful construction of finite dimensional approximations with respect to space and angle allows us to extend 
the main arguments of these analyses also to the time-dependent radiative transfer equation.

As a first step towards the systematic numerical approximation,
we investigate a class of Galerkin semi-discretizations in space and angle,
and we formulate some basic conditions on the approximation spaces that enable us to prove uniform energy estimates, convergence, and exponential stability of the semi-discrete problems.
We also discuss in some detail a particular realization based on a spherical harmonics expansion in angle and a mixed finite element approximation in space that and present some numerical tests for this approximation later on. 
In a second step, we then discuss the time discretization by one-step methods. 
We investigate in detail the backward Euler scheme, for which we establish
convergence estimates, uniform discrete energy bounds, and exponential stability 
with similar arguments as on the continuous and semi-discrete level.
%
%
%

\medskip 

The remainder of the manuscript is organized as follows: 
In Section~\ref{sec:prelim}, we introduce our basic notations and assumptions. 
Section~\ref{sec:anal} contains a detailed definition of the problems under investigation
and summarizes some basic results about their well-posedness and the long-term behavior of solutions. 
In Section~\ref{sec:variational}, we propose and analyze variational formulations for the stationary and instationary radiative transfer problems which serve as our starting point for the construction of numerical approximations.
In Section~\ref{sec:stath}, we recall some basic results about the Galerkin approximation of the 
mixed variational formulation for the stationary problem. 
Section~\ref{sec:semi} is concerned with the corresponding Galerkin semi-discretization
for the time-dependent problem, for which we establish well-posedness, convergence, and exponential stability.
A particular semi-discretization consisting of a combination of a truncated spherical harmonics expansion in angle and a mixed finite element approximation in space is discussed in Section~\ref{sec:pnfem}. 
Section~\ref{sec:time} is then concerned with the time discretization by the backward Euler method.
The performance of the mixed $P_N$-finite element time stepping scheme is illustrated in Section~\ref{sec:num} by  numerical tests.

\section{Preliminaries} \label{sec:prelim}

\subsection{Geometric setting}

Throughout the manuscript, $\R \subset \RR^3$ denotes some bounded Lipschitz domain and
$\S^2 = \{ \s \in \RR^3: |\s|=1\}$ is the unit sphere. 
Let $\D = \R \times \S^2$ be the corresponding tensor product domain in space and angle.
The boundary $\partial \D = \partial \R \times \S^2$ can be split into two parts
\begin{align}
\Gamma_\pm := \{ (\r,\s) \in \partial\D :  \pm \n(\r) \cdot \s > 0\},
\end{align}
called the \emph{outflow} and \emph{inflow boundary}, respectively. 
Here and below $\n=\n(\r)$ denotes the outward unit normal vector at the point 
$r \in \partial\R$. 

\subsection{Function spaces}

The basic function space for our analysis of the radiative transfer problem 
will be the space $\VV = L^2(\D)$ of square integrable functions over $\D$
with scalar product 
\begin{align*}
(v,w)_\D = \int_{\R \times \S^2} v(\r,\s) w(\r,\s) d(\r,\s)
\end{align*}
and norm $\|v\|_{L^2(\D)} = (v,v)_\D^{1/2}$. 
We use the symbol $(\cdot,\cdot)$ also to denote the integral of products of functions over other sets.
Of particular importance for our considerations will be the space
\begin{align}
 \WW &= \{ v \in L^2(\D) : \sgrad v \in L^2(\D) \text{ and } |\sn|^{1/2} v \in L^2(\partial\D)\},
\end{align}
which consists of functions with square integrable directional derivatives and boundary traces. 
This space is equipped with the graph norm given by
$$
\|v\|^2_{\WW} = \|v\|^2_{L^2(\D)} + \|\sgrad v\|_{L^2(\D)}^2 + \||\sn|^{1/2} v\|_{L^2(\partial\D)}^2.
$$ 
We denote by $\WW_0 = \{v \in \WW : v|_{\Gamma_-}=0\}$ the subspace of functions with 
vanishing traces on the inflow boundary.
All these spaces are Hilbert spaces when equipped with their respective norms. 


For functions $\phi,\psi \in \WW$, the following integration-by-parts formula holds true
\begin{align} \label{eq:ibp}
(\sgrad \phi, \psi)_\D = -(\phi,\sgrad \psi)_\D + (\sn \phi, \psi)_{\partial\D}.
\end{align}
This is obvious for smooth functions and follows for the general case by a density argument \cite{EggerSchlottbom12}.


By $C^k([0,T];X)$ we denote the space of $k$-times continuously differentiable functions $\phi : [0,T] \to X$ with values in some Hilbert space $X$ and norm defined by $\|\phi\|_{C^k([0,T];X)} = \max_{j \le k} \|\partial_t^j \phi\|_{C^0([0,T];X)}$ and $\|\phi\|_{C^0([0,T];X)} = \max_{0 \le t \le T} \|\phi(t)\|_X$.
We will also utilize the Bochner spaces $L^p(0,T;X)$ and $H^k(0,T;X)$, which can be defined by density 
and which are equipped with their natural norms; see \cite{Evans98} for details and further properties of these spaces.

\subsection{Assumptions on the parameters}
Let $k \in L^\infty(\R \times [-1,1])$ and $\sigma_t \in L^\infty(\R)$ denote the scattering kernel and the total cross-section, respectively. We further denote by
$\sigma_s(\r)=\int_{\S^2} k(\r,\s \cdot \s') d\s$ and $\sigma_a(\r)=\sigma_t(\r)-\sigma_s(\r)$
the corresponding scattering and absorption coefficients. 
Throughout the manuscript, we assume that
\begin{enumerate}\itemsep0.3em
 \item[(A1)] $0 \le k(\r,\s \cdot \s') \le \overline k $ for all $\r \in \R$ and $\s,\s' \in \S^2$; 
 \item[(A2)] $0 < \underline{\sigma}_a \le \sigma_a(\r) \le \overline \sigma_a$ for all $\r \in \R$.
\end{enumerate}
These conditions are physically reasonable and allow us to present our main results without notational difficulties; they could even be relaxed to some extent \cite{EggerSchlottbom12,EggerSchlottbom14}.
Under assumption (A1) and (A2) the collision operator is self-adjoint, bounded, and coercive on $L^2(\D)$, i.e., $(\C \phi, \psi)_\D = (\phi, \C \psi)_\D$, and moreover
\begin{align} \label{eq:cprop}
 \|\C \phi\|_{L^2(\D)} \le \overline \sigma_t \|\phi\|_{L^2(\D)} 
\qquad \text{and} \qquad 
(\C \phi, \phi)_\D \ge \underline{\sigma}_a \|\phi\|^2_{L^2(\D)}
\end{align}
with $\underline\sigma_a>0$ and $\overline\sigma_t = \overline\sigma_a + 8 \pi \overline k< \infty$; see \cite{EggerSchlottbom12,ManResSta00} for elementary proofs.

\section{The radiative transfer problem} \label{sec:anal} \setcounter{equation}{0}

We assume that the domain $\R$ is surrounded by vacuum. 
Using the above notation, the radiative transfer problem under investigation 
can be written in compact form as 
\begin{align}
\partial_t \phi + \s \cdot\nabla \phi + \C \phi &= q, && \text{on } \D  \times (0,T),\label{eq:rte1}\\
\phi &= 0, && \text{on } \Gamma_- \times (0,T),  \label{eq:rte2} \\
\phi|_{t=0} &= \phi_0, && \text{on } \D.\label{eq:rte3}
\end{align}

Under assumptions (A1)--(A2), the existence and uniqueness of solutions for this initial 
boundary value problem can be proven via semigroup theory; see e.g. \cite[Ch~XXI, Par~2, Thm~3]{DautrayLions6}.
\begin{lemma} \label{lem:instat}
Let (A1)--(A2) hold. Then for any $q\in C^1([0,T],L^2(\D))$ and $\phi_0 \in \WW_0$  
there exists a unique (classical) solution $\phi\in C^1([0,T],L^2(\D)) \cap C^0([0,T],\WW_0)$ 
to \eqref{eq:rte1}--\eqref{eq:rte3}, and
 \begin{align*}
  \|\phi\|_{C^1([0,T];L^2(\D))} + \|\sgrad \phi\|_{C^0([0,T];L^2(\D))} 
   \leq C (\|\phi_0\|_{\WW} + \|q\|_{C^1([0,T],L^2(\D))}).
 \end{align*}
\end{lemma}

If the source $q=\widehat q$ is independent of time, then $\phi(t)$ can be shown to converge to a steady state $\widehat\phi$ which is characterized by the stationary problem 
\begin{align}
 \s \cdot\nabla \widehat \phi + \C \widehat \phi &= \widehat q, && \text{on } \D, \label{eq:rte1s}\\
 \widehat\phi &= 0, && \text{on } \Gamma_-.  \label{eq:rte2s}
\end{align}
Assumptions (A1)--(A2) are sufficient to show that this stationary problem is again uniquely solvable; 
see \cite[Chapter~XXI, Par~2, Thm~4]{DautrayLions6} or \cite[Thm~3.1]{EggerSchlottbom12} for proofs.
\begin{lemma} \label{lem:stat}
Let (A1)--(A2) hold. Then for any $\widehat q\in L^2(\D)$, the problem \eqref{eq:rte1s}--\eqref{eq:rte2s} has 
a unique (strong) solution $\widehat\phi\in \WW_0$, and
$\|\widehat \phi\|_{\WW} \leq C \|\widehat q\|_{L^2(\D)}$
with  $C$ independent of the data. 
\end{lemma}
We use the hat symbol here to designate functions that are independent of time. 
It is not difficult to show the exponential stability of the instationary problem which implies exponential convergence to equilibrium. From the energy estimate \eqref{eq:ee1} given below, we obtain
\begin{lemma} \label{lem:exp}
Let (A1)--(A2) hold. Furthermore, let $\widehat\phi$ denote the solution of \eqref{eq:rte1s}--\eqref{eq:rte2s}, and let $\phi$ be the solution of \eqref{eq:rte1}--\eqref{eq:rte3} with $q=\widehat q \in L^2(\D)$. 
Then
 \begin{align*}
  \| \phi(t)-\widehat \phi\|^2_{L^2(\D)} \leq e^{-\underline \sigma_a t} \|\phi_0-\widehat \phi\|^2_{L^2(\D)}.
 \end{align*}
\end{lemma}
\begin{remark} \label{rem:guideline}
Existence and uniqueness of solutions, uniform in time estimates, 
and convergence to a well-defined steady state will also be our guiding 
principles for the design and the analysis of numerical methods in the following sections. 
\end{remark}

\section{Variational formulations}\label{sec:variational} \setcounter{equation}{0}

Our discretization strategy for the time-dependent radiative transfer problem is based on an
appropriate variational formulation for the stationary problem \cite{EggerSchlottbom12}.
For convenience of the reader, we recall the basic steps.
The starting point is to split the function $\phi$ into even and odd parts with respect to the angular variable.
Following \cite{Vladimirov61}, we define for any $\phi \in L^2(\S^2)$
\begin{align} \label{eq:split}
 \phi^\pm(\s) = \frac{1}{2}(\phi(\s) \pm \phi(-\s)),
\end{align}
called the even and odd parities of $\phi$, respectively. 
Note that $\phi=\phi^+ + \phi^-$ holds by construction.
Corresponding splittings can be defined also for functions depending on additional variables and they 
induce natural decompositions of the function spaces 
\begin{align} \label{eq:splitspace}
\VV = \VV^+ \oplus \VV^- \qquad \text{and} \qquad \WW = \WW^+ \oplus \WW^-
\end{align}
which are orthogonal with respect to $L^2(\D)$. 
Moreover, it is easy to see that
\begin{align} \label{eq:map}
\sgrad \phi^\pm \in \VV^\mp 
\qquad \text{and} \qquad 
\C \psi^\pm \in \VV^\pm
\end{align}
for any $\phi \in \WW$ and $\psi \in \VV$, respectively;  we refer to \cite{EggerSchlottbom12,Vladimirov61} for details and further properties.

\subsection{The stationary problem}

Let us recall the mixed variational principle for the 
stationary problem  \eqref{eq:rte1s}--\eqref{eq:rte2s} presented in \cite{EggerSchlottbom12}. 
Using the orthogonal splitting into even and odd parts, 
the radiative transfer problem characterizing the steady state can be written equivalently as 
\begin{align*}
\sgrad \widehat\phi^- + \C \widehat \phi^+ = \widehat q^+ 
\qquad \text{and} \qquad 
\sgrad \widehat\phi^+ + \C \widehat \phi^- = \widehat q^-
\end{align*}
with boundary condition $\widehat \phi = \widehat \phi^+ + \widehat \phi^- = 0$ on $\Gamma_-$.
Any strong solution $\widehat\phi \in \WW_0$ of the stationary problem \eqref{eq:rte1s}--\eqref{eq:rte2s} 
can then be seen to solve also the following mixed variational problem.
\begin{problem}[Stationary problem] \label{prob:stat}
Given $\widehat q \in L^2(\D)$, find $\widehat\phi = \widehat \phi^+ + \widehat \phi^- \in \WW^+ \oplus \VV^-$ such that 
\begin{align} 
-(\widehat \phi^-,\sgrad \widehat\psi^+)_\D + (|\sn| \widehat \phi^+, \widehat\psi^+)_{\partial\D} + (\C \widehat \phi^+, \widehat\psi^+ )_\D &= (\widehat q^+, \widehat\psi^+)_\D, \label{eq:stat1}\\
(\sgrad \widehat \phi^+, \widehat\psi^-)_\D + (\C \widehat \phi^-, \widehat\psi^-)_\D &=  (\widehat q^-, \widehat\psi^-)_\D, \label{eq:stat2}
\end{align}
for all test functions $\widehat\psi^+ \in \WW^+$ and $\widehat\psi^- \in \VV^-$.
\end{problem}
The hat symbols are used again to emphasize that the functions are independent of time.
Note that the coupling between the two equations is due to the presence of the directional 
derivatives $\s \cdot \nabla$ which shift even to odd functions. 
\begin{remark}
Any solution $\widehat \phi \in \WW_0$ of \eqref{eq:rte1s}--\eqref{eq:rte2s} also solves Problem~\ref{prob:stat},
which follows by testing the equation \eqref{eq:rte1s} with appropriate test functions $\widehat\psi^+$ and $\widehat\psi^-$, and some elementary manipulations.
The main step in the derivation of the variational principle is to observe that 
\begin{align*}
(\sgrad \widehat \phi^-, \widehat\psi^+)_\D 
&= -(\widehat \phi^-,\sgrad \widehat\psi^+)_\D + (\sn \widehat \phi^-, \widehat\psi^+)_{\partial\D}. 
\end{align*}
Since $\sn \widehat \phi^-$ and $\widehat\psi^+$ are both even functions, one  obtains 
\begin{align*}
(\sn \widehat \phi^-, \widehat\psi^+)_{\partial\D} = 2 (\sn \widehat \phi^-, \widehat\psi^+)_{\Gamma_-}.
\end{align*}
From the boundary condition \eqref{eq:rte2s}, one can deduce that $\sn \widehat\phi^- = - \sn \widehat\phi^+ = |\sn| \widehat\phi^+$ on $\Gamma_-$, which allows to replace the boundary terms appropriately.  
As a consequence of the derivation, one can see that the boundary conditions are incorporated naturally here.
For details of the derivation, see  \cite{EggerSchlottbom12}.
\end{remark}

\begin{remark}
The function spaces in the weak formulation are chosen with minimal regularity such that all terms are well-defined. 
Since two different spaces $\WW^+$ and $\VV^-$ are involved, we call Problem~\ref{prob:stat}
a \emph{mixed variational formulation}.
The \emph{energy space} $\WW^+ \oplus \VV^-$ for the problem consists of functions with mixed regularity. It is a Hilbert space when equipped with the induced norm $\|\phi\|_{\WW^+ \oplus \VV^-}^2 = \|\phi^+\|_\WW^2 + \|\phi^-\|_\VV^2$. 
\end{remark}

The well-posedness of Problem~\ref{prob:stat} now follows almost directly from the results 
of the previous section.

\begin{lemma}[Well-posedness] \label{lem:statweak} 
Let (A1)--(A2) hold. 
Then for any $\widehat q \in L^2(\D)$, Problem~\ref{prob:stat} has a unique solution 
$\widehat \phi \in  \WW^+ \oplus \VV^-$ 
and $\|\widehat \phi\|_{\WW^+ \oplus \VV^-} \le C \|\widehat q\|_{L^2(\D)}$. 
In addition, $\widehat \phi^- \in \WW^-$ with uniform bound $\|\widehat \phi^-\|_\WW \le C' \|\widehat q\|_{L^2(\D)}$, 
and $\widehat \phi$ is the unique solution of  \eqref{eq:rte1s}--\eqref{eq:rte2s}.
\end{lemma}
\begin{proof}
Existence of a solution $\widehat \phi \in \WW_0$ follows from Lemma~\ref{lem:stat}.
To show the uniqueness also in the larger space $\WW^+ \oplus \VV^-$, 
we test \eqref{eq:stat1}--\eqref{eq:stat2} 
with $\widehat \psi^+ = \widehat \phi^+$ and $\widehat \psi^-=\widehat \phi^-$, 
which yields
\begin{align*}
(\widehat q,\widehat \phi) 
&= (\widehat q^+,\widehat \phi^+)_\D + (\widehat q^-,\widehat \phi^-)_\D \\ 
&= (|\sn| \widehat \phi^+, \widehat \phi^+)_{\Gamma_+} + (\C \widehat \phi,\widehat \phi)_\D \ge \underline \sigma_a \|\widehat \phi\|^2_{\partial\D}.
\end{align*}
Here we used the coercivity property \eqref{eq:cprop} of the collision operator. 
Uniqueness now follows from this a-priori estimate and linearity of the problem.
\end{proof}
\begin{remark}
A slight modification of the proof would allow to obtain a more general existence result under 
weaker assumptions on the parameters and the data; we refer to \cite{EggerSchlottbom12} for details. 
\end{remark}

\subsection{The instationary problem}

Let us now turn to the time-dependent radiative transfer problem.
Similar as above, one can verify that any classical solution $\phi \in C^1([0,T];\VV) \cap C^0([0,T];\WW_0)$ of the instationary problem \eqref{eq:rte1}--\eqref{eq:rte3} 
also solves the following variational problem.
\begin{problem}[Instationary problem] \label{prob:instat} $ $\\
Find $\phi = \phi^+ + \phi^- \in L^2(0,T;\WW^+ \oplus \VV^-) \cap H^1(0,T;(\WW^+)' \oplus \VV^-)$ with
 $\phi(0) = \phi_0$ and such that 
\begin{align}
\langle\partial_t \phi^+, \widehat\psi^+\rangle_\D -( \phi^-,\sgrad \widehat\psi^+)_\D + (|\sn|  \phi^+, \widehat\psi^+)_{\partial\D} + (\C  \phi^+, \widehat\psi^+ )_\D &= (q^+, \widehat\psi^+)_\D, \label{eq:instat1}\\
(\partial_t \phi^-, \widehat\psi^-)_\D + (\sgrad  \phi^+, \widehat\psi^-)_\D + (\C  \phi^-, \widehat\psi^-)_\D &=  (q^-, \widehat\psi^-)_\D, \label{eq:instat2}
\end{align} 
for all $\widehat\psi^+ \in \WW^+$ and $\widehat\psi^- \in \VV^-$ and almost every $t \in (0,T)$.
\end{problem}
\begin{remark}
Here $(\WW^+)'$ denotes the dual space of $\WW^+$ and the term $\langle\partial_t \phi^+, \psi^+\rangle_\D$ is the duality product between $(\WW^+)'$ and $\WW^+$.  
%
%
Note that by the choice of the function spaces all terms are well-defined and also the initial conditions make sense.
\end{remark}

Existence of a solution to Problem~\ref{prob:instat} for a smooth source function $q$ follows directly from Lemma~\ref{lem:instat} and uniqueness is a consequence of the following energy estimates. 

\begin{lemma}[Energy estimate] \label{lem:ee} $ $\\
Let $\phi$ be a solution of Problem~\ref{prob:instat} with $\phi_0 \in L^2(\D)$ and $q\in L^2(0,T;L^2(\D))$. 
Then 
\begin{align}  \label{eq:ee1}
\|\phi(t)\|_{L^2(\D)}^2 \le  e^{-\underline\sigma_a t} \|\phi_0\|^2_{L^2(\D)} + \frac{1}{\underline\sigma_a} \int_0^t e^{-\underline\sigma_a (t-t')}\|q(t')\|^2_{L^2(\D)} dt'
\end{align}
holds for all $0\leq t\leq T$. 
If in addition $\phi_0 \in \WW_0$ and $q\in H^1(0,T;L^2(\D))$, then
\begin{align*}
\|\partial_t \phi(t)\|_{L^2(\D)}^2 \le e^{-\underline\sigma_a t} \|q(0) - \C \phi_0 - \sgrad \phi_0\|^2_{L^2(\D)} + \frac{1}{\underline\sigma_a} \int_0^t e^{-\underline\sigma_a (t-t')}\|\partial_t q(t')\|^2_{L^2(\D)} dt'
\end{align*}
and $\|\sgrad \phi^+(t)\|_{L^2(\D)} \le C \big( \|q\|_{H^1(0,T;L^2(\D))} + \|\phi_0\|_\WW \big)$ with $C$ independent of the data.
\end{lemma}
\begin{proof}
Testing \eqref{eq:instat1}--\eqref{eq:instat2} with $\widehat\psi^+=\phi^+(t)$ and $\widehat\psi^-=\phi^-(t)$, and summing up, yields 
\begin{align*}
&\tfrac{1}{2} \tfrac{d}{dt} \|\phi(t)\|^2_{L^2(\D)} + \||\sn| \phi^+(t)\|^2_{L^2(\partial\D)} + (\C \phi(t), \phi(t))_{\D}  \\
&\qquad \qquad = (q(t), \phi(t))_{\D} \le \tfrac{\underline \sigma_a}{2} \|\phi(t)\|^2_{L^2(\D)} + \tfrac{1}{2 \underline\sigma_a}\|q(t)\|^2_{L^2(\D)},
\end{align*}
where we used Young's inequality in the last step. 
Using the coercivity of the collision operator, 
the third term in the first line can be estimated from below by $(\C \phi,\phi)_{\D} \ge \underline \sigma_a \|\phi\|^2_{L^2(\D)}$. 
The estimate \eqref{eq:ee1} then follows from the Gronwall lemma.
To show the second estimate, we first formally differentiate \eqref{eq:instat1}--\eqref{eq:instat2} with respect to time, and then test with $\widehat\psi^+ = \partial_t \phi^+(t)$ and $\widehat\psi^-=\partial_t\phi^-(t)$. The estimate for the time derivative now follows with the same arguments as the first one. 
The last estimate follows from the previous ones via elementary manipulations.
\end{proof}

Using the a-priori estimates and a density argument, we obtain 
the following well-posedness result for Problem~\ref{prob:instat}, 
which yields a slight generalization of Lemma~\ref{lem:instat}.
\begin{lemma}[Well-posedness] \label{lem:instatweak}
For any $q \in H^1(0,T;L^2(\D))$ and $\phi_0 \in \WW_0$, 
Problem \ref{prob:instat} has a unique solution $\phi\in L^2(0,T;\WW^+\oplus\VV^-)\cap H^1(0,T;\VV)$, 
for which the conclusions of Lemma~\ref{lem:ee} hold true. \\
If $q \in C^1([0,T];L^2(\D))$, then $\phi$ coincides with the classical solution of \eqref{eq:rte1}--\eqref{eq:rte3} given by Lemma~\ref{lem:instat}. 
\end{lemma}
\begin{proof}
For $q \in C^1(0,T;L^2(\D))$, the existence of a solution follows from Lemma~\ref{lem:instat}. 
Due to the estimate of Lemma~\ref{lem:ee} and linearity of the problem, 
we can construct a solution for $q \in H^1(0,T;L^2(\D))$ by density. 
Uniqueness and the a-priori bounds for the general case follow from linearity of the problem and 
a denisty argument.
\end{proof}

As a direct consequence of the energy estimates of Lemma~\ref{lem:ee}, 
we also obtain the exponential stability of the instationary radiative transfer problem.
\begin{lemma}[Exponential stability] \label{lem:expweak}
Let (A1)--(A2) hold and let $\widehat \phi = \widehat\phi^+ +\widehat \phi^-$ and $ \phi = \phi^++\phi^-$ denote, respectively, the solutions of Problems~\ref{prob:stat} and \ref{prob:instat} with $q=\widehat q\in L^2(\D)$ and $\phi_0 \in \WW_0$.
Then
 \begin{align*}
  \| \phi(t)-\widehat \phi\|^2_{L^2(\D)} \leq e^{-\underline\sigma_a t} \|\phi_0-\widehat \phi\|^2_{L^2(\D)}.
 \end{align*}
\end{lemma}
\begin{proof}
The difference $\widetilde \phi(t) = \widehat \phi - \phi(t)$ satisfies  Problem~\ref{prob:instat} with right hand side $q \equiv 0$ and initial value $\widetilde \phi(0)=\widehat \phi - \phi_0$. 
The assertion then follows readily from the energy estimate \eqref{eq:ee1} of Lemma~\ref{lem:ee}.
\end{proof}

\subsection{Operator formulations} \label{sec:operator}

To highlight the structure of the time-dependent radiative transfer problem, 
let us shortly discuss an equivalent operator formulation. 
Let the collision operator $\C$ be given as above and define
the transport operator $\A : \WW^+ \to \VV^-$ and $\R : \WW^+ \to (\WW^+)'$ by
\begin{align*}
\A \phi^+ = \sgrad \phi^+
\quad \text{and} \quad 
\langle \R \phi^+, \psi^+ \rangle_\D= (|\sn|\phi^+, \psi^+)_{\partial\D}.
\end{align*}
The adjoint $\A':\VV^-\to (\WW^+)'$ of $\A$ is defined by $\langle \A' \phi^-, \psi^+\rangle_\D := (\phi^-,\A \psi^+ )_\D$. 
In the definition of the two operators $\R$ and $\A'$, the left hand side has 
to be understood as duality product between the spaces $(\WW^+)'$ and $\WW^+$.
As already mentioned in the introduction, Problem~\ref{prob:instat} can then be written 
as an abstract evolution equation
\begin{align*}
\partial_t \begin{pmatrix} \phi^+ \\ \phi^-\end{pmatrix} 
+ \begin{pmatrix} \C + \R & -\A' \\ \A & \C \end{pmatrix} \begin{pmatrix} \phi^+ \\ \phi^-\end{pmatrix} 
= 
\begin{pmatrix} q^+ \\ q^- \end{pmatrix} \quad \text{for} \quad t>0,
\end{align*}
with initial conditions $(\phi^+(0),\phi^-(0))=(\phi_0^+,\phi^-_0)$.
Note that the collision operator $\C$ has a slightly different meaning in the two lines.
The equality in the above system has to be understood in the sense of $(\WW^+)' \times \VV^-$.
A corresponding formulation for the stationary problem \cite{EggerSchlottbom12} can be obtained by omitting the time derivatives. 
The operator governing the  evolution problem can be split into two parts
\begin{align*}
\begin{pmatrix} 0 & -\A' \\ \A & 0 \end{pmatrix}
\quad \text{and} \quad 
\begin{pmatrix} \C + \R & 0 \\ 0 & \C \end{pmatrix}
\end{align*} 
which are skew-adjoint and dissipative, respectively. 
This basic structure of the radiative transfer problem
will be a main guideline for the design and analysis of numerical methods below.

\begin{remark} \label{rem:SH}
Identifying $\WW^+ \times \VV^-$ with $\WW^+ \oplus \VV^-$ 
allows us to rewrite the system equivalently as 
\begin{align*}
\partial_t \phi(t) + (\S+\H) \phi(t) &= q(t), \quad t>0, 
\end{align*}
with initial condition $\phi(0)=\phi_0$. 
This equation has to be understood in the sense of $(\WW^+ \otimes \VV^-)'$
and the two operators $\S$ and $\H$ are defined by 
\begin{align*}
\langle \S \phi,\psi\rangle_\D &= (\sgrad \phi^+,\psi^-)_\D - (\phi^-, \sgrad \psi^+)_\D, 
\qquad \text{and} \\
\langle \H \phi,\psi \rangle_\D &= (\C \phi, \psi)_\D + (|\sn| \phi^+,\psi^+)_{\partial\D}.
\end{align*}
Note that $\S$ and $\H$ are again skew-adjoint and dissipative, respectively.
Let us emphasize that, in contrast to the original form \eqref{eq:rte1}--\eqref{eq:rte3} of the radiative transfer problem, the directional derivative operator is understood in a distributional sense here. 
Moreover, the boundary conditions are incorporated naturally in the definition of the operators.
\end{remark}

\section{Galerkin discretization of the stationary problem}\label{sec:stath} \setcounter{equation}{0}

We start by recalling the basic discretization strategy for the stationary problem proposed in \cite{EggerSchlottbom12}.
Let $\WW_h^+ \subset \WW^+$ and $\VV^-_h \subset \VV^-$ denote some finite dimensional subspaces. 
As approximation for Problem~\ref{prob:stat}, we then consider the following Galerkin method.
\begin{problem}[Galerkin discretization] \label{prob:stath}
Find $\widehat \phi_h \in \WW_h^+ \oplus \VV_h^-$ such that 
\begin{align}
 (|\sn| \widehat \phi_h^+, \widehat\psi_h^+)_{\partial\D} - ( \widehat\phi_h^-,\sgrad \widehat\psi_h^+)_\D  + (\C \widehat \phi_h^+, \widehat\psi_h^+ )_\D &= (\widehat q^+, \widehat\psi_h^+)_\D,  \label{eq:stath1}\\
 (\sgrad  \widehat\phi_h^+, \widehat\psi_h^-)_\D + (\C  \widehat \phi_h^-, \widehat\psi_h^-)_\D &=  (\widehat q^-, \widehat\psi_h^-)_\D, \label{eq:stath2}
\end{align} 
for all test functions $\widehat\psi_h^+ \in \WW^+$ and $\widehat\psi_h^- \in \WW^-$.
\end{problem}

\begin{remark} \label{rem:B}
\noindent
For our analysis it will be convenient to rewrite the discrete stationary Problem~\ref{prob:stath} 
in the following more compact form:
Find $\widehat \phi_h \in \WW_h^+ \oplus \VV_h^-$ such that
\begin{align} \label{eq:statop}
B(\widehat \phi_h; \widehat \psi_h) &= l(\widehat \psi_h) \qquad \text{for all } \widehat\psi_h \in \WW_h^+ \oplus \VV_h^-
\end{align}
with bilinear form 
\begin{align*}
B(\widehat \phi; \widehat \psi) 
&= (|\sn| \widehat \phi^+, \widehat\psi^+)_{\partial\D} + (\C \widehat \phi, \widehat\psi )_\D  - ( \widehat\phi^-,\sgrad \widehat\psi^+)_\D
+ (\sgrad  \widehat\phi^+, \widehat\psi^-)_\D 
\end{align*}
and linear form defined by 
\begin{align*}
l(\widehat \psi) = (\widehat q,\widehat\psi)_\D.
\end{align*}
Note that $B$ is just the bilinear form associated with the operator $\S+\H$ defined in Remark~\ref{rem:SH}.
\end{remark}
%
In order to guarantee existence and uniqueness of a solution and uniform a-priori estimates, 
some compatibility condition for the discretization spaces $\WW_h^+$ and $\VV_h^-$ is required. 
We will therefore assume in the following that
\begin{enumerate} \itemsep0.3em
 \item[(A3)] $\WW_h^+ \subset \WW^+$ and $\VV_h^- \subset \VV^-$ are finite dimensional;
 \item[(A4)] $\sgrad \WW_h^+ \subset \VV_h^-$. 
\end{enumerate}
Assumption (A3) is natural for conforming Galerkin approximation and the compatibility condition (A4)
can easily be realized by explicit construction. One possibility is to choose $\WW_h^- := \sgrad \WW_h^+$
and another choice will be discussed in Section~\ref{sec:pnfem} below. 

\begin{lemma}  \label{lem:Bprop}
Let (A1)--(A4) hold.
Then the bilinear form $B(\cdot;\cdot)$ satisfies:
\begin{enumerate} \itemsep0.3em
\item[(i)] $B(\widehat \phi; \widehat \psi) \le C_B \|\widehat \phi\|_{\WW^+ \oplus \VV^-} \|\widehat \psi\|_{\WW^+ \oplus \VV^-}$ for all $\widehat \phi, \widehat \psi^+ \in \WW^+ \oplus \VV^-$ with $C_B=\overline\sigma_t$;
\item[(ii)] $\sup_{\widehat \psi_h \in \WW_h^+ \oplus \VV_h^-} B(\widehat \phi_h; \widehat \psi_h)/\|\widehat \psi_h\|_{\WW^+ \oplus \VV^-} \ge \beta \|\widehat \phi_h\|_{\WW^+ \oplus \VV^-}$ for all $\widehat \phi_h\in \WW_h^+ \oplus \VV_h^-$ with stability constant $\beta=\underline\sigma_a/(1+\overline\sigma_t^2)$;
\item[(iii)] $\sup_{\widehat \phi_h\in \WW_h^+ \oplus \VV_h^-} B(\widehat \phi_h; \widehat \psi_h) > 0$ for all $\widehat \psi_h \in \WW_h^+ \oplus \VV_h^-$ with $\widehat \psi_h \ne 0$.
\end{enumerate}
\end{lemma}
\begin{proof}
Continuity (i) of the bilinear form $B$ follows  from the Cauchy-Schwarz inequality 
and the definition of the norms. 
The stability condition (ii) can be proven by choosing the test function $\widehat \psi_h = \widehat \phi_h + \alpha \sgrad \widehat \phi_h^+$ with $\alpha=2 \underline\sigma_a/(1+\overline\sigma_t^2)$.  Due to assumption (A4), such a choice is possible. 
Condition (iii) is shown in a similar manner as (ii).
Let us refer to \cite[Pro~3.3]{EggerSchlottbom12} for the proof of a slightly more general result.
\end{proof}

Based on the assertions of the previous lemma, we now readily obtain 

\begin{lemma}[Well-posedness and quasi-optimality] $ $ \label{lem:stath}\\
Let (A1)--(A4) hold and $\widehat q\in L^2(\D)$.
Then Problem~\ref{prob:stath} admits a unique solution $\widehat \phi_h$ and 
\begin{align*}
\|\widehat \phi - \widehat \phi_h\|_{\WW^+ \oplus \VV^-} \le C \inf_{\widehat \psi_h \in \WW_h^+ \oplus \VV_h^-} \|\widehat \phi - \widehat \psi_h\|_{\WW^+ \oplus \VV^-}
\end{align*}
with $C=C_B / \beta$ and constants $C_B$ and $\beta$ as in the previous lemma.
\end{lemma}
\begin{proof}
The assertions follow from the Babuska-Aziz lemma \cite{Babuska71,BabuskaAziz72} and the precise form of the stability constant was derived in \cite{XuZikatanov03}. 
\end{proof}

\begin{remark}
Since $B(\widehat \phi;\widehat  \phi) \ge (\C \widehat \phi, \widehat \phi)_\D \ge \underline\sigma_a \|\widehat \phi\|_{L^2(\D)}^2$, 
the discrete problem \eqref{eq:statop} is uniquely solvable for any choice of finite dimensional subspaces 
$\WW_h^+ \subset \WW^+$ and $\VV_h^- \subset \VV^-$. 
The compatibility condition (A4) is therefore only needed to obtain the uniform a-priori estimate for the solution in the stronger norm.
\end{remark}

\begin{remark}
The solution $\widehat\phi_h \in \WW_h^+ \oplus \VV_h^-$ of Problem~\ref{prob:stath}
depends linearly and continuously on the solution $\widehat \phi \in \WW^+ \oplus \VV^-$ of Problem~\ref{prob:stat}. 
Using standard terminology, the operator $\Pi_h : \widehat \phi \mapsto \widehat \phi_h$ defined by 
\begin{align} \label{eq:elliptic_projection}
B(\widehat \Pi_h \widehat \phi; \widehat \psi_h) = B(\widehat \phi; \widehat \psi_h) \qquad \text{for all } \widehat\psi_h \in \WW_h^+ \oplus \VV_h^-
\end{align}
will be called \emph{elliptic projection}. 
By the estimates of Lemma~\ref{lem:stath}, we have 
\begin{align} \label{eq:ellipticest}
\|\widehat \Pi_h \widehat \phi - \widehat \phi\|_{\WW^+ \oplus \VV^-} \le C \inf_{\widehat \psi_h \in \WW_h^+ \oplus \VV_h^-} \|\widehat\phi - \widehat\psi_h\|_{\WW^+ \oplus \VV^-}.
\end{align}
This quasi-optimal approximation property plays an important role in the numerical analysis of time-dependent problems; see \cite{DouglasDupontWheeler78,Dupont73} for similar arguments in the context of more standard problems.
\end{remark}

\section{Semi-discretization of the instationary problem} \label{sec:semi} \setcounter{equation}{0}

Using the discretization strategy for the stationary problem, one directly  
obtains a corresponding Galerkin semi-discretization for Problem~\ref{prob:instat}, which reads
\begin{problem}[Galerkin semi-discretization] \label{prob:instath} $ $\\
Find $\phi_h \in H^1(0,T;\WW_h^+ \oplus \VV_h^-)$ with initial value
$\phi_h(0) = \Pi_h \phi_0$ and such that
\begin{align*}
(\partial_t \phi_h^+, \widehat\psi_h^+)_\D -( \phi_h^-,\sgrad \widehat\psi_h^+)_\D + (|\sn|  \phi_h^+, \widehat\psi_h^+)_{\partial\D} + (\C  \phi_h^+, \widehat\psi_h^+ )_\D &= (q^+, \widehat\psi_h^+)_\D, \\
(\partial_t \phi_h^-, \widehat\psi_h^-)_\D + (\sgrad  \phi_h^+, \widehat\psi_h^-)_\D + (\C  \phi_h^-, \widehat\psi_h^-)_\D &=  (q^-, \widehat\psi_h^-)_\D, 
\end{align*} 
holds for all $\widehat\psi_h^+ \in \WW_h^+$ and $\widehat\psi_h^- \in \WW_h^-$ and for almost every $t \in (0,T)$.
\end{problem}
The hat symbol is used again to emphasize that the test functions are independent of time.
Since $\WW_h^+$ and $\VV_h^-$ are finite dimensional, this problem amounts to a linear system of ordinary differential equations, and one directly obtains

\begin{lemma} \label{lem:instath}
Let $q \in L^2(0,T;L^2(\D))$ and $\phi_0 \in \WW^+ \oplus \VV^-$ and assume that (A1)--(A3) hold. 
Then Problem~\ref{prob:instath} admits a unique solution $\phi_h \in H^1(0,T;\WW_h^+ \oplus \VV_h^-)$ and
\begin{align} \label{eq:ee1h}
\|\phi_h(t)\|^2_{L^2(\D)} \le e^{-\underline \sigma_a t} \|\phi_0\|_{L^2(\D)}^2 + \frac{1}{\underline{\sigma}_a}\int_0^t e^{-\underline \sigma_a (t-s)} \|q(t')\|_{L^2(\D)}^2 dt'.
\end{align}
\end{lemma}
\begin{proof}
Existence and uniqueness follow from the Picard-Lindel\"of theorem and the energy estimate is obtained with the same arguments as that of Lemma~\ref{lem:ee}.
\end{proof}

\begin{remark} 
In comparison with the continuous problem, we could slightly relax the conditions on the regularity of the data here.
Moreover, the assumption (A4), which was required for the well-posedness of the discrete stationary problem, could be dropped here.  
\end{remark}
As a direct consequence of the discrete energy estimate \eqref{eq:ee1h}, 
one can also guarantee exponential convergence to the discrete equilibrium similarly as on the continuous level.

\begin{lemma}[Exponential stability] \label{lem:eeh} $ $\\
Let the assumptions (A1)--(A4) hold. Moreover, let $\widehat \phi_h \in \WW_h^+ \oplus \VV_h^-$ and $\phi_h \in H^1(0,T;\WW_h^+ \oplus \VV_h^-)$ 
denote the solutions of Problem~\ref{prob:stath} and Problem~\ref{prob:instath} with $q=\widehat q \in L^2(\D)$ and  $\phi_0 \in \WW^+\oplus\WW^-$, respectively. 
Then 
\begin{align*}
\|\phi_h(t) - \widehat \phi_h\|^2_{L^2(\D)} \le  e^{-\underline{\sigma}_a t} \|\phi_h(0) - \widehat \phi_h\|^2_{L^2(\D)}.
\end{align*}
\end{lemma}
%
The previous results show that the proposed Galerkin semi-discretization of the time-dependent radiative transfer problem 
inherits all good stability properties from the continuous level. 

\bigskip 

The error analysis for the Galerkin semi-discretization can now be conducted with standard arguments. 
For convenience of the reader, we summarize the most basic result and provide a short proof.

\begin{theorem}[Error estimate for the semi-discretization] \label{thm:res1} $ $\\
Let (A1)--(A4) hold and let $\phi$ and $\phi_h$ denote the solutions of Problems~\ref{prob:instat} and \ref{prob:instath}, respectively. 
Then
\begin{align*} 
\|\phi(t) - \phi_h(t)\|_{L^2(\D)}
\le C \big( \inf_{\widehat\psi_h} \|\phi_0 - \widehat\psi_h\|_{\WW^+ \oplus \VV^-}
+ \inf_{\widehat\psi_h} \|\phi(t) - \widehat\psi_h\|_{\WW^+ \oplus \VV^-} \\
+ \int_0^t \inf_{\widehat\psi_h} \|\partial_t \phi(t') - \widehat\psi_h\|_{\WW^+ \oplus \VV^-} \; dt' \big).
\end{align*}
The infimum is taken over all functions $\widehat\psi_h \in \WW_h^+ \oplus \VV_h^-$ both times.
\end{theorem}
\begin{proof}
With the notation of Section~\ref{sec:operator} and Remark~\ref{rem:B}, we can rewrite Problem~\ref{prob:instat} as 
\begin{align*} 
( \partial_t \phi(t), \psi) + B(\phi(t);\psi) = l(t;\psi), \quad\text{for all } \widehat\psi \in \WW^+\oplus\VV^-,\  t>0, 
\end{align*}
with $l(t;\widehat \psi) = (q(t),\widehat \psi)_\D$.
The corresponding semi-discrete problem reads
\begin{align*} 
( \partial_t \phi_h(t), \psi_h) + B(\phi_h(t);\psi_h) = l(t;\psi_h), \quad\text{for all } \widehat\psi_h \in \WW^+\oplus\VV^-,\  t>0.
\end{align*}
Using the elliptic projection defined in \eqref{eq:elliptic_projection}, the error can be decomposed as
\begin{align} \label{eq:est1}
\|\phi(t) - \phi_h(t)\|_{L^2(\D)} \le \|\phi(t) - \Pi_h \phi(t)\|_{L^2(\D)} + \|\Pi_h \phi(t) - \phi_h(t)\|_{L^2(\D)}
\end{align}
into an approximation error and a discrete error component.
The first part can be estimated by the approximation properties of the elliptic projection \eqref{eq:ellipticest} yielding
\begin{align*}
\|\phi(t) - \Pi_h \phi(t)\|_{L^2(\D)} 
&\le \|\phi(t) - \Pi_h \phi(t)\|_{\WW^+ \oplus \VV^-}
 \le C \inf_{\widehat\psi_h \in \WW_h \oplus \VV_h^-} \|\phi(t) - \widehat\psi_h\|_{\WW^+ \oplus \VV^-}.
\end{align*}
Furthermore, the discrete error component can be seen to satisfy 
\begin{align*}
&( \partial_t \Pi_h \phi(t) - \partial_t \phi_h(t), \psi_h ) + B(\Pi_h \phi(t) - \phi_h(t); \psi_h)  \\
&\qquad \qquad = (\partial_t \Pi_h \phi(t) - \partial_t \phi(t),\psi_h) + B(\Pi_h \phi(t)-\phi(t);\psi_h)
\\ &\qquad \qquad  
= ( \partial_t \Pi_h \phi(t) - \partial_t \phi(t),\psi_h).
\end{align*}
In the last step we used here the particular definition of the elliptic projection.
From the discrete energy estimate given in Lemma~\ref{lem:instath}, we then readily obtain
\begin{align*}
 \|\Pi_h \phi(t) - \phi_h(t)\|_{L^2(\D)} \le C \big( \|\Pi_h \phi(0) - \phi_0\|_{L^2(\D)} + \int_0^t   \|\partial_t \Pi_h \phi(t') - \partial_t \phi(t')\|_{L^2(\D)} dt' \big).
\end{align*}
Since the elliptic projection commutes with the time derivative, the assertion of the theorem now 
follows from \eqref{eq:ellipticest} similarly as the bound for the first term in \eqref{eq:est1} above.
\end{proof}

\begin{remark}
The estimate of Theorem~\ref{thm:res1} seems to be somewhat sub-optimal concerning the regularity requirements for the solution. This is typical for hyperbolic problems; see e.g. \cite{DouglasDupontWheeler78,Dupont73}.
We also used a crude estimate $\|\phi(t) - \Pi_h \phi(t)\|_{L^2(\D)} \le \|\phi(t) - \Pi_h \phi(t)\|_{\WW^+ \oplus \VV^-}$ for the elliptic projection. 
This seems to be unavoidable here, since, in contrast to more standard elliptic problems, only very weak regularity results are available for the stationary radiative transfer equation; let us refer to \cite{deVorePetrova00,GoLiPeSe88} for details.
\end{remark}

\section{A mixed PN finite element scheme} \label{sec:pnfem} \setcounter{equation}{0}

We now discuss a particular choice of approximation spaces $\WW_h^+$ and $\VV_h^-$
that satisfies the basic assumptions (A3)--(A4) needed for our analysis.
The definition of the domain $\D = \R \times \S^2$ by a tensor product 
suggests to utilize a similar tensor product construction also for the approximation spaces. 
We begin with the approximation of the angular variable.

\subsection{Spherical harmonics expansion}

The spherical harmonics $\Y_l^m$ with $-l \le m \le l$ and $l \ge 0$,
form a complex valued complete orthonormal system of $L^2(\S^2)$, see e.g. \cite{LewisMiller84} for an analytic expression of $Y_l^m$.
Any square integrable function $\phi \in L^2(\D)$ can therefore be expanded into a Fourier series
\begin{align*}
\phi(\r,\s) = \sum_{l=0}^\infty \sum_{m=-l}^l \phi_l^m (\r) \Y_l^m(\s) \qquad \text{with} \qquad \phi_l^m(\r) = \int_{\S^2} \phi(\r,\s) \overline{ \Y_l^m(\s)} d\s.
\end{align*}
Moreover, the even and odd parities of $\phi$ can be expressed as
\begin{align*}
\phi^+(\r,\s) = \sum_{l=0}^\infty \sum_{m=-2l}^{2l} \phi_{2l}^m (\r) \Y_{2l}^m(\s)
\quad \text{and} \quad 
\phi^-(\r,\s) = \sum_{l=0}^\infty \sum_{m=-2l-1}^{2l+1} \phi_{2l+1}^m (\r) \Y_{2l+1}^m(\s).
\end{align*}
By Parseval's theorem, one has $\|\phi\|_{L^2(\D)}^2 = \sum_{l=0}^\infty \sum_{m=-l}^l \|\phi_l^m\|^2_{L^2(\R)}$ 
and thus the Fourier coefficients of a function $\phi \in L^2(\D)$ satisfy 
$\|\phi_l^m\|_{L^2(\R)} \le \|\phi\|_{L^2(\D)} < \infty$. 
Unfortunately, the precise characterization of the spatial smoothness of Fourier coefficients of 
functions $\phi \in \WW$ is more involved.
We can however give at least a simple sufficient condition.
\begin{lemma} \label{lem:smooth}
Let $\phi \in L^2(\D)$. Then the following assertions hold:

\smallskip

(i) If $\sum_{l=0}^\infty \sum_{m=-l}^l \|\phi_l^m\|_{H^1(\R)}^2 < \infty$, then $\phi \in \WW$.

\smallskip

(ii) If $\sum_{l=0}^\infty \sum_{m=-2l}^{2l} \|\phi_{2l}^m\|_{H^1(\R)}^2<\infty$, then $\phi^+ \in \WW^+$.
\end{lemma}
\begin{proof}
The statement follows by observing that $\|\sgrad \phi\|_{L^2(\D)} \le \|\nabla \phi\|_{L^2(\D)}$ 
and using the Parseval identity to express the norm on the right hand side. 
\end{proof}

\begin{remark}
Spatial regularity and sufficient decay of the Fourier coefficients therefore is a sufficient criterion 
to guarantee the directional smoothness of functions in $L^2(\D)$. 
\end{remark}

\subsection{PN-approximation}


A truncation of the spherical harmonics expansion yields a natural discretization for the angular variable. 
Let us define 
\begin{align*}
\VV_{N} := \Big\{\psi(\r,\s)  = \sum_{l=0}^N \sum_{m=-l}^l \psi_l^m(\r) \Y_l^m(\s) : \psi_l^m \in L^2(\R)\Big\}.
\end{align*} 
and further set $\VV_N^-=\VV_N \cap \VV^-$ and $\WW_N^+:= \VV_{N} \cap \WW^+$.
As noted above, the spatial smoothness of the Fourier coefficients implies the directional
 smoothness of the function $\phi$.
Since this property is essential for the space discretization, we state it explicitly.
\begin{lemma} \label{lem:smoothN}
Let $\phi_N=\sum_{l=0}^N  \sum_{m=-l}^l \phi_l^m(\r) \Y_l^m(s)$. Then

\smallskip

(i) $\phi_N \in \VV_N$ if, and only if, $\phi_l^m \in L^2(\R)$ for $-l \le m \le l$ and $0 \le l \le N$.  

\smallskip

(ii) If $\phi_{2l}^m \in H^1(\R)$ for $-2l \le m \le 2l$ and $0 \le 2l \le N$, then  $\phi_N^+ \in \WW_N^+$.
\end{lemma}

The following result now follows directly from the mapping properties of the operators $\A$ and $\C$ stated in \eqref{eq:map} and yields a criterion for the appropriate choice of the order $N$ in the $P_N$-approximation.
\begin{lemma} \label{lem:mappingN}
Let $\VV_N^\pm$ and $\WW_N^+$ be defined as above. Then there holds:

(i) $\C \VV_N^+ \subset \VV_N^+$ and $\C \VV_N^- \subset \VV_N^-$.

\smallskip

(ii) If $N$ is odd, then $\A \WW_N^+ \subset \VV_N^-$.
\end{lemma}

\begin{remark}
The second condition allows to verify assumption (A4), which was required to establish the well-posedness of the Galerkin approximations. Let us note that therefore all results of Sections~\ref{sec:semi} and \ref{sec:time} apply for the $P_N$ approximation provided $N$ is odd, although $\WW_N^+$ and $\VV_N^-$ are still infinite dimensional spaces at this point; see \cite{EggerSchlottbom12} for details.
\end{remark}

\subsection{Spatial discretization}

In order to obtain finite dimensional approximation spaces $\WW_h^+$ and $\VV_h^-$,
we still have to choose an appropriate spatial discretization for the Fourier coefficients
in the spherical harmonics expansion. 
We choose here two different spaces $\XX_h^+$ and $\XX_h^-$ for the approximation of the even and odd Fourier coefficients. To satisfy the conditions of Lemma~\ref{lem:smoothN} and Lemma~\ref{lem:mappingN},
we require that
\begin{enumerate} \itemsep0.3em
 \item[(A5)] $\XX_h^+ \subset H^1(\R)$ and $ \XX_h^- \subset L^2(\R)$;
 \item[(A6)] $\nabla \XX_h^+ \subset (\XX_h^-)^d$;
\end{enumerate}
As we will see below, these conditions can be verified without much difficulty in practice.
The fully discrete approximation spaces for the even and odd component of the solution are then defined as
\begin{align}
\WW_{h}^+ &= \Big\{ \psi(\r,\s) = \sum_{2l=0}^N \sum_{m=-2l}^{2l} \psi_{2l,h}^m(\r)\Y_{2l}^m(\s) : \psi_{2l,h}^m \in \XX_h^+ \Big\} 
\qquad \mbox{and}   \label{eq:spacesWhN}\\
\VV_{h}^- &= \Big\{ \psi(\r,\s) = \sum_{2l+1=0}^{N} \sum_{m=-2l-1}^{2l+1} \psi_{2l+1}^m(\r)\Y_{2l+1}^m(\s) : \psi_{2l+1}^m \in \XX_h^- \Big\}.  \label{eq:spacesVhN}
\end{align}
As a direct consequence of the construction of the approximation spaces, 
we obtain a conforming Galerkin approximation which satisfies the required stability condition.
\begin{lemma}
Let (A5) and (A6) hold. Then 

\smallskip

(i) $\WW_{h}^+ \subset \WW^+$ and $\VV_{h}^- \subset \VV^-$.

\smallskip

(ii) If $N$ is odd, then in addition $\A\WW_{h}^+ \subset \VV_{h}^-$.

\smallskip 

\noindent
Consequently, $\WW_h^+$ and $\VV_h^-$ satisfy the assumptions (A3)--(A4) in case (A5)--(A6) hold and $N$ is odd.
\end{lemma}


As a final step of the construction of the semi-discretization, let us now discuss the construction of appropriate spaces $\XX_h^+$ and $\XX_h^-$ that satisfy the compatibility conditions (A5)--(A6).

\subsection{The PN-finite element method}\label{sec:PN_FEM}
Let $T_h$ be some regular partition of the domain $\R$ into tetrahedrons 
and denote by $P_k(T_h)$ the space of piecewise polynomials of maximal degree $k$. 
As approximation spaces for the Fourier coefficients, we then choose
\begin{align*}
 \XX_h^+ = P_{k+1}(T_h) \cap H^1(\R) 
\qquad \text{and} \qquad 
 \XX_h^- = P_k(T_h).
\end{align*}
With this construction, one now easily verifies

\begin{lemma}
Let $k \ge 0$ and set $\XX_h^+ = P_{k+1}(T_h) \cap H^1(\R)$ and $\XX_h^- = P_k(T_h)$. 
Then (A5) and (A6) hold. Moreover, for $N$ odd, the mixed $P_N$-finite element spaces $\WW_h^+$ and $\VV_h^-$
defined in \eqref{eq:spacesWhN}--\eqref{eq:spacesVhN} satisfy the assumptions (A3)--(A4) and all results of Section~\ref{sec:semi} hold true.
\end{lemma}


\begin{remark}
Since two different discrete spaces are involved, the resulting method is called a \emph{mixed $P_N$-finite element method}. The resulting discrete spaces $\WW_h^+$ and $\VV_h^-$ also have good approximation properties, 
i.e., spectral with respect to the order $N$ of the $P_N$ approximation and polynomial with respect to the mesh size $h$, provided that the function to be approximated is sufficiently smooth. Some explicit estimates will be discussed in Section~\ref{sec:num}.
\end{remark}

\begin{remark}
For $k=0$, the even Fourier coefficients are approximated by standard continuous piecewise linear finite elements 
and piecewise constant functions are used to approximate the odd coefficients. 
This will be the discretization used for our numerical tests below.
\end{remark}

\section{Time discretization}\label{sec:time} \setcounter{equation}{0}

As a last step in the discretization process, we now turn to the time discretization.
We only consider the implicit Euler method in detail here, but other one-step methods 
could be analyzed in a similar manner.
Let $\tau>0$ be the step size and define $t^n = n \tau$ for $n \ge 0$. We further denote by 
\begin{align*}
\bar\partial_\tau \phi^n := \frac{1}{\tau} \big( \phi^n - \phi^{n-1} \big) 
\end{align*}
the backward difference quotient which serves as approximation for the time derivative.
The fully discrete approximation of Problem~\ref{prob:instat} then reads as follows.
\begin{problem}[Full discretization] \label{prob:full}
Set $\phi_h^0 := \Pi_h \phi_0$ and for $n \ge 1$ find $\phi_h^{n} \in \WW_h^+ \oplus \WW_h^-$ such that
\begin{align} \label{eq:impeul}
(\bar\partial_\tau \phi_h^n, \widehat \psi_h)_\D + B(\phi_h^n;\widehat\psi_h) = l(t^{n};\widehat\psi_h), 
\end{align}
for all $\widehat\psi_h \in \WW_h^+ \oplus \VV_h^-$. 
Here $l(t;\widehat\psi) = (q(t),\widehat\psi)_\D$ and $B(\cdot;\cdot)$ is defined as in Remark~\ref{rem:B}. 
\end{problem}
From the results about the discrete stationary problem, we readily obtain
\begin{lemma}[Well-posedness] \label{lem:full_wellposed} 
Let $\phi_h^{n-1}$ and $q(t^n) \in L^2(\D)$ be given. 
Then for any $\tau>0$, the problem \eqref{eq:impeul} has a unique solution $\phi_h^n \in \WW_h^+ \oplus \VV_h^-$, 
i.e., Problem~\ref{prob:full} is well-posed.
\end{lemma}

With similar arguments as on the continuous level, we further obtain
\begin{lemma}[Discrete energy estimate] \label{lem:fullh}
Let (A1)--(A3) hold and $q \in H^1(0,T;L^2(\D))$. 
Then for any $0<\tau \le 1/(2\underline\sigma_a)$ 
the solution $\{\phi_h^n\}_{n \ge 0}$ of Problem~\ref{prob:full} satisfies
\begin{align}\label{eq:ee1fh}
 \|\phi_h^n\|^2_{L^2(\D)} \le  e^{-\underline\sigma_a t^n} \|\phi_0\|^2_{L^2(\D)} + \frac{2}{\underline\sigma_a}\sum_{k=1}^{n} \tau e^{-\underline\sigma_a (t^n-t^k)}\|q(t^{k})\|^2_{L^2(\D)}, \qquad n \ge 0.
\end{align}
\end{lemma}
Let us emphasize that the restriction of the time step is only needed for ease of notation. 
\begin{proof}
Note that $(\bar\partial_\tau \phi_h^n,\phi_h^n)_\D \ge \frac{1}{2\tau} \|\phi_h^n\|_\D^2 -\frac{1}{2\tau} \|\phi_h^{n-1}\|^2_\D$. 
Testing the discrete variational principle \eqref{eq:impeul} with $\widehat \psi_h = \phi_h^n$ 
and using the coercivity of the bilinear form $B$ then yields 
\begin{align*}
\frac{1}{2\tau} \|\phi_h^n\|^2 + \underline\sigma_a \|\phi_h^n\|^2_\D 
\le \frac{1}{2\tau} \|\phi_h^{n-1}\|^2_\D + \frac{1}{\underline\sigma_a} \|q(t^n)\|_\D^2 + \frac{\underline\sigma_a}{4} \|\phi_h^n\|^2_\D.  
\end{align*}
By elementary manipulations, we further obtain
\begin{align*}
\|\phi_h^n\|^2_\D \le \frac{1}{1+\frac{3}{2}\underline\sigma_a \tau } \|\phi_h^{n-1}\|^2_\D +  \frac{1}{1+\frac{3}{2}\underline\sigma_a\tau} \frac{2}{\underline\sigma_a} \|q(t^n)\|^2_\D.
\end{align*}
Since $\frac{1}{1+\frac{3}{2}\underline\sigma_a\tau} \le e^{-\underline\sigma_a \tau}$ for $0 < \tau \le 1/(2\underline\sigma_a)$, the assertion now follows by induction.
\end{proof}

Let us finally state the basic a-priori estimate for the fully discrete scheme.
We present a statement of the error estimate that covers a wide class of Galerkin schemes.

\begin{theorem}[Error estimate] \label{thm:res2}
Let (A1)--(A4) hold and $q\in H^2(0,T;L^2(\D))$. 
Then 
\begin{align*}
&\|\phi(t^n) - \phi_h^n\|_{L^2(\D)}^2 
\leq C \inf_{\widehat\psi_h \in \WW_h^+ \oplus \VV_h^-} \|\phi(t^n) - \widehat\psi_h\|^2_{\WW^+ \oplus \VV^-} +\\
& \qquad \qquad + C \sup_{0 \le t \le t^n} \big\{\inf_{\widehat\psi_h \in \WW_h^+ \oplus \VV_h^-} \|\partial_t \phi - \widehat\psi_h\|_{L^2(\D)}^2 \big\} +  C \tau^2 \sup_{0 \le t \le t^n} \|\partial_{tt} \phi(t)\|^2_{L^2(\D)}.
\end{align*}
\end{theorem}
\begin{proof}
The error can be split in the usual way into 
\begin{align*}
\|\phi(t^n) - \phi_h^n\|_\D \le \|\phi(t^n) - \Pi_h \phi(t^n)\|_\D + \|\Pi_h \phi(t^n) - \phi_h^n\|_\D.
\end{align*}
The first error component is readily estimated by \eqref{eq:ellipticest} yielding
\begin{align*}
\|\phi(t^n) - \Pi_h \phi(t^n)\|_\D 
\le C' \inf_{\widehat\psi_h \in \WW_h^+ \oplus \VV_h^-} \|\phi(t^n) - \widehat\psi_h\|_{\WW^+ \oplus \VV^-}.
\end{align*}
Using the variational characterizations of $\phi_h^n$ and $\phi$ as well as the properties of the elliptic projection, the second error component $w_h^n = \Pi_h \phi(t^n) - \phi_h^n$ can be seen to satisfy
 \begin{align*}
(\bar\partial_\tau w_h^n, \widehat \psi_h)_\D + B(w_h^n;\widehat \psi_h) 
&= (\bar\partial_\tau \Pi_h \phi(t^n) - \bar \partial_\tau \phi(t^n),\widehat \psi_h)_\D
  + (\bar\partial_\tau \phi(t^n) - \partial_t \phi(t^n),\widehat \psi_h)_\D.
\end{align*}
The discrete error component $w_h^n$ therefore satisfies the fully discrete problem with 
right hand side $Q(t^n)$ that can be estimated by
\begin{align*}
\|Q(t^n)\|_\D 
&\le \|\bar\partial_\tau \Pi_h \phi(t^n) - \bar \partial_\tau \phi(t^n)\|_\D + \|\bar\partial_\tau \phi(t^n) - \partial_t \phi(t^n)\|_\D \\
& \le 
C'  \inf_{\widehat\psi_h \in \WW_h^+ \oplus \VV_h^-} \|\partial_t \phi(\xi^n) - \widehat\psi_h\|_{\WW^+ \oplus \VV^-} 
+ C'' \tau \|\partial_{tt} \phi(\eta^n)\|_\D 
\end{align*}
with appropriate $\xi^n,\eta^n \in (t^{n-1},t^n)$. 
Here we again used the approximation property \eqref{eq:ellipticest} of the elliptic projection. 
The assertion of the theorem now follows from the discrete stability estimate \eqref{eq:ee1fh}.
Note that, similar to Lemma~\ref{lem:ee}, $q\in H^2(0,T;L^2(\D))$ implies $\partial_{tt} \phi \in L^\infty(0,T;L^2(\D))$.
\end{proof}

\begin{remark}
The error estimate of the theorem could be sharpened such that the terms in the second line contain exponentially decaying factors like in the a-priori estimate of Lemma~\ref{lem:fullh}. 
Let us emphasize that we did not have to specify the approximation spaces $\WW_h^+$ and $\VV_h^-$ in detail,
but only required the assumptions (A3) and (A4). The error estimate thus covers a wide class of Galerkin approximations.
\end{remark}

\begin{remark}
When using the mixed $P_N$-finite-element discretization discussed in the previous section, we can expect 
spectral convergence of the error with respect to $N$ and polynomial with respect to the mesh size $h$, 
provided the solution $\phi$ is sufficiently smooth, i.e., its Fourier coefficients are spatially smooth and decay sufficiently fast.Some explicit estimates with respect to $N$ and $h$ will be provided for a particular example in the next section.
\end{remark}

\section{Numerical results} \label{sec:num} \setcounter{equation}{0}

We now illustrate the theoretical results derived in the previous sections by some numerical tests. 

\subsection{A test problem}

For ease of presentation, we consider a problem where the solution is homogeneous in one spatial direction and, therefore, 
can be represented by a function $\phi(\r,\s,t)$ with $\r=(\r_1,\r_3) \in \RR^2$. Note that this corresponds to a fully three dimensional setting with a certain symmetry. Therefore, we still have $\s \in \S^2\subset\RR^3$ for the directions of propagation. 
Many benchmark problems have this structure.

For our tests, we set $\R=(0,1)\times (0,1)$ and consider the function
\begin{align*}
  \phi(\r,\s,t) = (1-e^{-t})\sum_{l=0}^{{N_{\text{max}}}}\frac{\phi_0^0 (\r)}{(l+1)^2} \Y_l^0(\s) \qquad \text{with} \qquad \phi_0^0(\r) = \sin(\pi \r_1)\sin(\pi\r_3)
\end{align*}
as the analytic solution. 
One can easily verify that the homogeneous boundary condition $\phi(\r,\s,t)=0$ on $\Gamma_-$ holds and that $\phi_0(\r,\s)=\phi(\r,\s,0)=0$ at time $t=0$.
The model parameters are chosen as 
\[
  k(r,\s\cdot \s')=\frac{1}{4\pi},\quad \mu_s(\r)=1, \quad \mu_a(\r)=\frac{1}{100},
\]
and the source term $q(\r,\s,t)$ is defined such that $\phi(\r,\s,t)$ satisfies the radiative transfer equation \eqref{eq:rte1}.


\subsection{Discretization and implementation}

In all our numerical tests, we use the mixed $P_N$-finite element method described in Section~\ref{sec:PN_FEM} 
with polynomial order $k=0$ and different moment orders $N$ odd. 
The Fourier coefficients with even index are therefore approximated by piecewise linear finite elements, 
while the odd coefficients are approximated by piecewise constants. 
For the time discretization we use the implicit Euler method as described in Section~\ref{sec:time}.

%
In every step \eqref{eq:impeul} of the implicit Euler method, a large linear system  has to be solved. 
This can be done efficiently by preconditioned iterative solvers. Since the $P_N$-finite element method 
is based on a tensor-product construction, the application of the system matrix can be realized very efficiently; 
let us refer to \cite{ArridgeEggerSchlottbom13} for details on the implementation

\subsection{Numerical results}
For illustration of the estimates of Theorem~\ref{thm:res2}, we report about the errors
\begin{align*}
 e_h^{\pm}  &= \max_{n} \|\phi^\pm(t^n)-\phi^{n,\pm}_h\|_{L^2(\D)} \qquad \text{and}\\
 (E_h^{+})^2 &= \max_{n} \|\sgrad\phi^+(t^n)-\sgrad\phi^{n,+}_h\|_{L^2(\D)}^2 + \|\phi^+(t^n)-\phi^{n,+}_h\|_{L^2(\D)}^2.
\end{align*}
Since different approximation spaces are used for the even and odd component, we report separately on the 
even and odd error contributions.
According to the construction of the numerical scheme, the errors basically consist of three components: 
\begin{enumerate}\itemsep0.3em
\item[(i)] the truncation error in the Fourier series depending on $N$ and the decay of the moments $\phi_l^m$;
\item[(ii)] the spatial approximation error due to the finite element approximation depending on the  smoothness of the moments $\phi_{l}^m$ and $\partial_t\phi_{l}^m$; 
\item[(iii)] the error due to the time discretization.
\end{enumerate}
With our numerical tests, we try to illuminate separately these three error contributions.

\bigskip

\textbf{Test 1: Convergence in $N$.} 
For the exact solution, we choose $N_{\text{max}}=40$. The spatial domain is partitioned by a triangulation 
with mesh size $h=1/128$ resulting in $16\,129$ vertices and $31\,752$ triangles. The time step is chosen to be $\tau=1/1000$. We then compute the numerical solutions by the $P_N$-finite element method 
with polynomial order $k=0$ and for different moment orders $N\in \{1,3,5,7\}$ in the $P_N$ approximation. The results of our tests are depicted in Table~\ref{tab:conv_N}. 


\begin{table}[ht!]
 \centering
 \begin{tabular}{c | r r |r r |r r|r }
  $N$ & $e^{+}_h$ & (eoc)&  $E_h^{+}$ & (eoc)& $e^{-}_h$& (eoc)& dofs\\
  \hline
  1	&6.06e-02	 &--		&6.11e-02	 &--		&3.60e-02&--	&    111\,385 \\
  3	&2.41e-02	 &(0.84)	&2.54e-02	 &(0.80)	&1.77e-02&(0.64)&    414\,294\\
  5	&1.35e-02	 &(1.13)	&1.57e-02	 &(0.95)	&1.10e-02&(0.93)&    908\,727 \\
  7	&8.85e-03	 &(1.26)	&1.19e-02	 &(0.82)	&7.74e-03&(1.04)& 1\,594\,684
 \end{tabular}
 \medskip
 \caption{Test~1: Error measures $e_h^{\pm}$ and $E_h^{+}$ for different $N$ together with the estimated order of convergence, and the number of degrees of freedom. \label{tab:conv_N}}
  \vspace*{-1em}
\end{table}


We observe that all error components decay approximately like $O(1/N)$, which is what we expect since the truncation error $\sum_{l=N+1}^{N_{max}} 1/(l+1)^2 \approx 1/N$. 
We also can see that for this choice discretization parameters, the discretization error is dominated by the truncation of the spherical harmonics expansion.

\bigskip

\textbf{Test 2: Convergence with $h$.} 
We now choose $N_{\text{max}}=2$ for the exact solution. Note that the source term $q$ then has moments up to order $N_{\text{max}}+1=3$ and we therefore fix $N=3$ for the numerical simulations.
The time step is defined as $\tau=10^{-4}$.
By this construction, the discretization errors due to truncation of the spherical harmonics expansion and the time discretization will be negligible.
We then compute the numerical solutions with the $P_N$-finite element method for a sequence of uniformly refined meshes.
Since the exact solution has smooth moments $\phi_l^m$, we expect to observe convergence rates 
in accordance with standard interpolation error estimates, i.e., $e^+_h + e^-_h=O(h)$.
The results of our numerical tests are listed in Table~\ref{tab:conv_h}.

\begin{table}[ht!]
 \centering
 \begin{tabular}{c | r r |r r |r r|r }
  $1/h$ & $e^{+}_h$ & (eoc)&  $E_h^{+}$ & (eoc)& $e^{-}_h$& (eoc)& dofs\\
  \hline
   8	&1.94e-02	 &--		&1.67e-01	 &--		&4.02e-02&--	& 1,014\\   
  16	&3.77e-03	 &(2.36)	&7.19e-02	 &(1.21)	&1.90e-02&(1.08)& 5\,270\\    
  32	&8.37e-04	 &(2.17)	&3.36e-02	 &(1.10)	&9.07e-03&(1.07)&23\,766\\    
  64 	&2.02e-04	 &(2.05)	&1.63e-02	 &(1.05)	&4.42e-03&(1.04)&100\,694\\    
  128 	&5.42e-05	 &(1.90)	&8.03e-03	 &(1.02)	&2.18e-03&(1.02)&414\,294\\
 \end{tabular}
  \medskip
 \caption{Test 2: Error measures $e_h^{\pm}$ and $E_h^{+}$ for different mesh sizes together with the estimated order of convergence. \label{tab:conv_h}}
 \vspace*{-1em}
\end{table}
%

In accordance with Theorem~\ref{thm:res2}, we observe that $e_h^++e_h^- = O(h)$. 
Our numerical tests also reveal that $E_h^+=O(h)$ and $e_h^+=O(h^2)$.
These latter estimates could also be explained theoretically by similar arguments 
as used for mixed approximations of similar elliptic, parabolic, and hyperbolic problems.

\bigskip

\textbf{Test 3: Convergence in time.} 
As in the previous test, we choose $N_{\text{max}}=2$ for the analytic solution, 
and we fix $N=3$ for the angular approximation. 
Moreover, we use a uniform triangulation with mesh size $h=1/256$ corresponding to $65\, 025$ vertices and $127\;008$ triangles. This yields a total number of $1\,680\,470$ degrees of freedom (dofs) for the $P_N$-FEM approximation with polynomial order $k=0$.
We then compute the discrete solutions for different time step sizes $\tau\in\{1/2,1/4,1/8,1/16\}$.
The results of our numerical tests are displayed in Table~\ref{tab:conv_tau}. 

\begin{table}[ht!]
 \centering
 \begin{tabular}{c | r r |r r |r r|}
  $1/\tau$ & $e^{+}_h$ & (eoc)&  $E_h^{+}$ & (eoc)& $e^{-}_h$& (eoc)\\
  \hline
   2	&3.81e-02	 &--		&5.03e-02	 &--		&1.95e-02&--	\\
   4	&2.10e-02	 &(0.86)	&2.86e-02	 &(0.81)	&1.11e-02&(0.82)\\
   8	&1.11e-02	 &(0.92)	&1.55e-02	 &(0.89)	&6.00e-03&(0.88)\\
  16	&5.73e-03	 &(0.95)	&8.37e-03	 &(0.89)	&3.24e-03&(0.89)  \\ 
 \end{tabular}
  \medskip
 \caption{Test 3: Error measures $e_h^{\pm}$ and $E_h^{+}$ for different time step sizes together with the estimated order of convergence.  \label{tab:conv_tau}}
 \vspace*{-2em}
\end{table}
%
%
%

For these relatively large time steps, the time discretization error dominates the errors due to the angular and spatial approximation and we observe convergence of the form $O(\tau)$ which is what we can expect in terms of  Theorem~\ref{thm:res2}.

\section{Conclusions} \label{sec:disc} \setcounter{equation}{0}

Building on the mixed variational framework introduced in \cite{EggerSchlottbom12}, we proposed and analyzed a Galerkin framework for the systematic construction of numerical approximation schemes for time-dependent radiative transfer problems. 
A key idea of these discretization schemes is to approximate the even and odd parts of the solution in different finite dimensional spaces.

Using some general abstract conditions on the approximation spaces, we could provide a complete analysis of the resulting Galerkin semi-discretizations including stability estimates and convergence estimates for the discretization error. As a particular example for an appropriate Galerkin discretization, 
we discussed a mixed $P_N$-finite element method.

We also investigated the time discretization by the implicit Euler method and 
showed that the fully discrete scheme inherits all good stability properties from the continuous and the semi-discrete level. 

The analysis and error estimates presented in the paper are motivated by the fact, that the mixed variational formulation of the evolution problem has strong similarities with wave propagation problems. 
Therefore, also further aspects, like the derivation of explicit methods or the formulation of absorbing boundary conditions might be successfully extended to radiative transfer problems as well. 
Such generalizations are currently investigated by the authors.

\section{Acknowledgements}
The work of HE was supported by DFG via grants IRTG~1529, GSC~233, and TRR~154, and
MS acknowledges support by ERC via grant EU FP 7 - ERC Consolidator Grant 615216 LifeInverse.

\bibliographystyle{plain} 
\bibliography{tdrte}

\end{document}